\documentclass{amsart}
\usepackage{amssymb}
\usepackage{graphicx}
\theoremstyle{plain}
\newtheorem{theorem}{Theorem}
\newtheorem*{lemma}{Lemma}
\begin{document}

\title{Derivatives of Blaschke Products and Model Space Functions}

\author{David Protas}
\address{Department of Mathematics \\
         California State University \\
         Northridge, California 91330}
\email{david.protas@csun.edu}

\date{}

\begin{abstract}
The relationship between the distribution of zeros of an infinite Blaschke product $B$ and the inclusion in weighted Bergman spaces $A_{\alpha}^p$ of the derivative of $B$ or the derivative of functions in its model space $H^2  \ominus BH^2$ is investigated.
\end{abstract}

\maketitle

\section{Preliminaries}
If $f$ is analytic in the open unit disc $U$ and $0 < p < \infty$, then
\[M_p(r; f) = \left\{\frac{1}{2\pi} \int_0^{2\pi} |f(re^{it})|^p\,dt\right\}^{1/p}\]
is defined for each positive $r < 1$. The Hardy space $H^p$ is the set of all functions $f$, analytic in $U$, for which $\|f\|_{H^p} = \sup_{0<r<1} M_p(r; f)$ is finite. Let $dA(z)$ denote Lebesgue area measure. If $f$ is analytic in $U$, $0 < p < \infty$ and $\alpha > -1$, then $f$ is said to be in the space $A_{\alpha}^p$ if 
\[\|f\|_{A_{\alpha}^p} = \left\{\frac{1}{\pi} \iint_U |f(re^{it})|^p (1-r)^{\alpha}\,dA(z)\right\}^{1/p}\] 
is finite. Put $A^p = A_0^p$.

If $\{a_n\}$ is a sequence of complex numbers such that $0<|a_n|<1$ for all $n = 1,2, \ldots$ and $\sum_{n = 1}^{\infty} (1-|a_n|) < \infty$, the  Blaschke product
\[B(z) = \prod_{n = 1}^{\infty} \frac{\bar{a}_n}{|a_n|} \frac{a_n - z}{1 - \bar{a}_nz}\]
is an analytic function in $U$ with zeros $\{a_n\}$. A sequence $\{a_n\}$ of points in $U$ is said to be separated or uniformly discrete if there is a constant $\delta > 0$ such that $\rho(a_m,a_n) \ge \delta$ for all $m \ne n$, where $\rho$ is the pseudohyperbolic metric in $U$ and is given by
\[ \rho(z,w) = \left|\frac{z-w}{1-\bar{w}z}\right|,\quad z, w \in U. \]
The sequence $\{a_n\}$ is said to be uniformly separated if there is a constant $\delta > 0$ such that
\[ \inf_n \prod_{m \ne n} \rho(a_m,a_n) \ge \delta. \]
A Blaschke product whose zeros are uniformly separated is called an interpolating Blaschke product. It is clear that uniformly separated sequences form a proper subset of the set of all uniformly discrete sequences.

For any Blaschke product $B$, let $(BH^2)^{\perp} = H^2 \ominus BH^2$ be the orthogonal complement of the invariant subspace $BH^2$ in $H^2$. $(BH^2)^{\perp}$ is called the \textit{model space} or \textit{star-invariant subspace} for $B$ in $H^2$. Here are a few standard results about $(BH^2)^{\perp}$ (see, for example, \cite{GMR} and \cite{HNP}) that we will be using. For $z \in U$, the reproducing kernel for $(BH^2)^{\perp}$ is
\[K_z(u) = \frac{1 - \overline{B(z)}B(u)}{1 - \overline{z}u},\]
$u \in U$. That is, $\langle f, K_z \rangle = f(z)$ for all $f \in (BH^2)^{\perp}$. Next, for any Blaschke product $B$ with zeros $\{a_n\}$, let $B_1 = 1$ and $B_n$ be the subproduct of $B$ with zeros $a_1, \ldots, a_{n-1}$, $n = 2, 3, \ldots$ . If
\[g_n(z) = B_n(z)(1 - |a_n|^2)^{1/2}/(1 - a_{n}z),\]
then $\{g_n\}$ is an orthonormal basis for $(BH^2)^{\perp}$. If the sequence $\{a_n\}$ is uniformly separated, then $\{h_n\}$ is a Riesz basis for $(BH^2)^{\perp}$, where
\[h_n(z) = (1 - |a_n|)^{1/2}/(1 - a_{n}z).\]

We will be considering the relationship between the condition $\sum_n (1-|a_n|)^{\beta} < \infty$ and the inclusion of $B'$ and the derivative of $(BH^2)^{\perp}$ functions in various spaces $A_{\alpha}^p$. In \cite{G}, A. Gluchoff proved that an inner function has its first derivative in $A_{\alpha}^p$, where $\alpha > -1$ and $p \ge \alpha + 2$, if and only if it is a finite Blaschke product. So, we will restrict our attention to $\alpha > -1$, $p < \alpha + 2$ since we are interested only in infinite Blaschke products.

In any theorem concerned with a function being in $A_{\alpha}^p$ for certain points $(\alpha,p)$, we will say that the \textit{scope} of the theorem is the set of all such points. Also, we write $K_1 \lesssim K_2$ or $K_2 \gtrsim K_1$ if there exists a constant $C>0$ such that $K_1 \le CK_2$ for all values of $K_1$ and $K_2$ under consideration, and we write $K_1 \asymp K_2$ if $K_1 \lesssim K_2$ and $K_1 \gtrsim K_2$.

\section{The Derivative of a Blaschke Product}

In \cite{K}, H.~O.~Kim proved that if $\alpha > -1$, $\max((\alpha + 2)/2, \alpha + 1) < p < \alpha +2$, and  $B$ is a Blaschke product with zeros $\{a_n\}$ such that 
\begin{equation} \label{sigma}
\sum_n (1-|a_n|)^{2-p+\alpha} < \infty,
\end{equation}
then $B' \in A^p_{\alpha}$. In the other direction, Gluchoff proved in \cite[Theorem 6]{G} that if $B$ is a Blaschke product  with zeros $\{a_n\}$ that are uniformly separated and if $B' \in A_{\alpha}^p$ where $\alpha > -1$, $p \ge 1$, and $\alpha + 1 < p < \alpha + 2$, then condition~\eqref{sigma} holds. In \cite[Theorem 3]{P2}, the scope of Gluchoff's result was extended to $p > 0$. On the other hand, in \cite[Theorem~2(ii)]{AV}, A. Aleman and D. Vukoti\'{c} generalized Gluchoff's result to uniformly discrete zeros and certain normal weights, and in \cite[Theorem~1]{PRR}, F.~P\'{e}rez-Gonz\'{a}lez, A.~Reijonen, and  J.~R\"{a}tty\"{a} further generalized this to doubling weights. Both of these last two generalizations are stated for $p > \frac{1}{2}$. However, the proofs of both can be seen to show that condition~\eqref{sigma} holds even with the hypothesis $p > \frac{1}{2}$ dropped. In the following theorem, we state the result for the classical weights $(1 - r)^{\alpha}$ that holds for uniformly discrete zeros and scope extended to $p > 0$, and we supply a straightforward proof.

\begin{theorem} \label{discrete} Let $B$ be an infinite Blaschke product with  zeros $\{a_n\}$ that are uniformly discrete. If $B' \in A_{\alpha}^p$ where $\alpha > -1$ and $\alpha + 1 < p < \alpha + 2$, then condition~\eqref{sigma} holds.
\end{theorem}
\begin{proof}
We have $\rho(a_i,a_j) \ge \delta > 0$ for all $j \ne i$. For each $n = 1, 2, \ldots$, put 
\[\Delta_n = \{z: \rho(z, a_n) < R\},\]
where $R =  \frac{\delta}{2}$. It is known (see \cite{DS}) that $\Delta_n$ is a Euclidean disk with radius equal to $R(1 - |a_n|^2)/(1  - R^2|a_n|^2)$, and that there exists a constant $C$ (independent of $n$) such that $1 - |z| \le C(1 - |a_n|)$ for all $z \in \Delta_n$. Also note that for each $n$, $z \in \Delta_n \Rightarrow |B(z)| \le |\frac{z - a_n}{1 - {\bar{a}_n}z}| < R$.
Thus,
\begin{align*}
 \iint_{\Delta_n} &\left(\frac{1 - |B(z)|}{1 - |z|} \right)^p (1-|z|)^{\alpha}\,dA(z)   \\
&\ge (1 - R)^p  \iint_{\Delta_n} (1 - |z|)^{\alpha -p}\,dA(z)  \\
&\ge (1 - R)^p C^{\alpha - p} (1 - |a_n|)^{\alpha - p}\pi \left( \frac{R(1 - |a_n|^2)}{1 - R^2 |a_n|^2} \right)^2  \\
&\asymp (1 - |a_n|)^{2 - p + \alpha}.
\end{align*}
Then since the disks $\Delta_n$ are pairwise disjoint,
\[ \sum_n (1 - |a_n|)^{2 - p + \alpha} \lesssim \iint_U\left(\frac{1 - |B(z)|}{1 - |z|} \right)^p (1-|z|)^{\alpha}\,dA(z). \]
The result follows since, as proved by P. Ahern in \cite{A2}, $B' \in A_{\alpha}^p$ if and only if
\[\iint_U \left(\frac{1 - |B(z)|}{1 - |z|} \right)^p (1-|z|)^{\alpha}\,dA(z) < \infty\]
when $\alpha > -1$, $p > \alpha + 1$.
\end{proof}

We now investigate what can be deduced when the hypothesis that the zeros of $B$ be uniformly discrete is dropped. In \cite[Theorem 6]{AC2}, P.~Ahern and D.~Clark proved that if $B$ is any Blaschke product with zeros $\{a_n\}$ such that $B' \in A_{\alpha}^1$ for $-1 < \alpha < -\frac{1}{2}$, then $\sum_n (1 - |a_n|)^{\beta} < \infty$ for all $\beta > \frac{1+\alpha}{-\alpha}$. This was generalized in \cite{P3} to $\sum_n (1 - |a_n|)^{\beta} < \infty$ for all $\beta > \frac{2 - p + \alpha}{p - \alpha - 1}$ if $-1 < \alpha < -\frac{1}{2}$, $\frac{3}{2} + \alpha < p \le 1$, and $B' \in A_{\alpha}^p$. We now prove a further generalization that increases the scope and is valid for $\beta = \frac{2 - p + \alpha}{p - \alpha - 1}$.

\begin{theorem} \label{indiscrete} Let $B$ be an infinite Blaschke product with  zeros $\{a_n\}$. If $B' \in A_{\alpha}^p$ where $\alpha > -1$ and $\frac{3}{2} + \alpha < p < 2 + \alpha$, then 
\[ \sum_n (1-|a_n|)^{\frac{2-p+\alpha}{p - 1 - \alpha}} < \infty.\]
\end{theorem}
\begin{proof}
First note that $0 < \frac{2 - p + \alpha}{p - \alpha - 1} < 1$ since $\alpha > -1$ and $\frac{3}{2} + \alpha < p < 2 + \alpha$. Suppose $B' \in A^p_{\alpha}$. By Theorem~5.1 of \cite{A1}, $B' \in A^1_{\alpha - p + 1}$ since $1 + \alpha < p < 2 + \alpha$, and then by Theorem~6.2 of \cite{A1}, $B' \in H^{p - 1 - \alpha}$ since $\frac{1}{2} < p - 1 - \alpha < 1$. Theorem 8 of \cite{AC1} then says that $ \sum_n (1-|a_n|)^{\frac{2-p+\alpha}{p - 1 - \alpha}} < \infty$, again since $\frac{1}{2} < p - 1 - \alpha < 1$.
\end{proof}

\section{The Derivative of Model Space Functions}

For a given infinite Blaschke product $B$, we will be investigating conditions that imply that $f' \in A^p_{\alpha}$ for all $f \in (BH^2)^{\perp}$. (In Theorem~4 of \cite{C}, W.~Cohn proved a result of this sort, but for $f'$ being in a Hardy space.) We will start, however, with a condition on $(\alpha, p)$ that ensures that  $f' \in A^p_{\alpha}$ for all $f \in H^2$.

\begin{theorem} \label{H2} Let $f$ be any function in $H^2$. If $\alpha > 1$ and $0 < p < \frac{4}{3} + \frac{2}{3}\alpha$, then $f' \in A^p_{\alpha}$.
\end{theorem}
\begin{proof}
Let $f \in H^2$. Assume for now that $p > 2$. By two theorems of Hardy and Littlewood, $M_p(r; f) \lesssim 1/(1 - r)^{\frac{1}{2} - \frac{1}{p}}$ (see \cite[Theorem~5.9]{D}) and then $M_p(r; f') \lesssim 1/(1 - r)^{\frac{3}{2} - \frac{1}{p}}$ (see \cite[Theorem~5.5]{D}). So, 
\[\|f'\|_{A_{\alpha}^p}^p \lesssim \int_0^1 (1 - r)^{-\frac{3p}{2} + 1 + \alpha}\, dr < \infty\]
for all $(\alpha, p)$ with $\alpha > 1$ and $2 < p < \frac{4}{3} + \frac{2}{3}\alpha$. Then, $\|f'\|_{A_{\alpha}^p}^p < \infty$
for all $(\alpha, p)$ with $\alpha > 1$ and $0 < p < \frac{4}{3} + \frac{2}{3}\alpha$ since the $A_{\alpha}^p$ spaces expand as $p$ decreases.
\end{proof}

We note that J.~Littlewood and R.~Paley proved the last result and its converse for $(\alpha, p) = (1, 2)$ in \cite{LP}. The next two theorems enable us to extend the region of points $(\alpha, p)$  where  $f \in (BH^2)^{\perp} \Rightarrow f' \in A^p_{\alpha}$ for every Blaschke product $B$, beyond the scope of Theorem \ref{H2}.

\begin{theorem} \label{perp1} Let $B$ be any Blaschke product. If $\alpha > -1$ and $0 < p < \frac{2}{3} + \frac{2}{3}\alpha$, then $f' \in A^p_{\alpha}$ for all $f \in (BH^2)^{\perp}$.
\end{theorem}
\begin{proof}
Let $f \in (BH^2)^{\perp}$. Then,
\[|f(z)| = |\langle f, K_z \rangle| \le \|f\|_{H^2}  \|K_z\|_{H^2} = \|f\|_{H^2} \left( \frac{1 - |B(z)|^2}{1 - |z|^2} \right)^{\frac{1}{2}}, \]
and so $|f(z)| \lesssim 1/(1 - |z|)^{\frac{1}{2}}$. Then $M_p(r; f) \lesssim 1/(1 - r)^{\frac{1}{2}}$, which implies by \cite[Theorem~5.9]{D}, that $M_p(r; f') \lesssim 1/(1 - r)^{\frac{3}{2}}$. Therefore, 
\[\|f'\|_{A_{\alpha}^p}^p \lesssim \int_0^1 (1 - r)^{-\frac{3p}{2} + \alpha}\, dr < \infty,\]
since $0 < p < \frac{2}{3} + \frac{2}{3}\alpha$.
\end{proof}

We now present a well known proposition (see \cite[Lemma~4.3]{M}, for example), which will be used a number of times in what follows.

\begin{lemma} \label{lem} Let $a \in U$, $\alpha > -1$, and $p > 0$. Then
\[ \iint_U \frac{(1 - |z|)^{\alpha}}{|1 - \bar{a}z|^{2p}}\,dA(z) \asymp
\begin{cases}
\hspace{24pt}1 & if \hspace{8pt} 0 < 2p < \alpha + 2 \\ 
\hspace{9pt}\log\frac{1}{1 - |a|} & if \hspace{8pt} 2p = \alpha + 2 \\
\frac{1}{(1 - |a|)^{2p - \alpha - 2}} & if \hspace{8pt} 2p > \alpha + 2
\end{cases}
\]
\end{lemma}

\begin{theorem} \label{perp2} Let $B$ be any Blaschke product. If $\alpha > 0$ and $0 < p < 1 + \frac{1}{2}\alpha$, then $f' \in A^p_{\alpha}$ for all $f \in (BH^2)^{\perp}$.
\end{theorem}
\begin{proof}
Let $f \in (BH^2)^{\perp}$. Assume for now that $p \ge 1$. $\{g_n\}$ is an orthonormal basis for $(BH^2)^{\perp}$, where $g_n(z) = B_n(z)(1 - |a_n|^2)^{1/2}/(1 - \bar{a}_{n}z)$ and $\{a_n\}$ is the sequence of zeros of $B$. So, 
\[f(z) = \sum_n c_n B_n(z) \frac{(1 - |a_n|^2)^{1/2}}{1 - \bar{a}_{n}z},\]
where $\sum_n |c_n|^2 = \|f\|_{H^2}^2 < \infty$. Then,
\begin{equation} \label{derivative}
f'(z) = \sum_n c_n B_n(z) \bar{a}_n \frac{(1 - |a_n|^2)^{1/2}}{(1 - \bar{a}_{n}z)^2} + \sum_n  c_n B'_n(z) \frac{(1 - |a_n|^2)^{1/2}}{1 - \bar{a}_{n}z}.
\end{equation}
Because we are assuming that $p \ge 1$, we can apply Minkowski's inequality to (\ref{derivative}). We get
\begin{multline} \label{ineq} \|f'\|_{A^p_{\alpha}} \lesssim \sum_n |c_n| (1 - |a_n|^2)^{1/2} \left\{\iint_U \frac{(1 - |z|)^{\alpha}}{|1 - \bar{a}_n z|^{2p}}\,dA(z) \right\}^{\frac{1}{p}} \\ + \sum_n |c_n| (1 - |a_n|^2)^{1/2} \left\{\iint_U \frac{(1 - |z|)^{\alpha - p}}{|1 - \bar{a}_n z|^p}\,dA(z) \right\}^{\frac{1}{p}}
\end{multline}
since $|B_n(z)| \le 1$ and $|B'_n(z)| \le 1/(1 - |z|)$ for all $n$. Let $I_1(a_n)$ be the first integral and $I_2(a_n)$ be the second integral in formula~(\ref{ineq}). By the Lemma, $I_1(a_n) \asymp 1$ since $\alpha >  0  > -1$ and $2p < \alpha + 2$, while $I_2(a_n) \asymp 1$ since $\alpha - p  > \frac{\alpha}{2} - p > -1$ and $p < \alpha - p + 2$. Thus,
\[\|f\|_{A^p_{\alpha}} \lesssim \sum_n |c_n| (1 - |a_n|^2)^{1/2} \le \left(\sum_n |c_n|^2 \right)^\frac{1}{2} \left(\sum_n (1 -|a_n|^2) \right)^\frac{1}{2} < \infty \]
by H\"{o}lder's inequality. This gives the result for $p \ge 1$. The result for $p > 0$ then follows immediately.
\end{proof}

Theorems \ref{H2}, \ref{perp1} and \ref{perp2} give us values of $(\alpha, p)$ at which the derivative of every model space function for every infinite Blaschke product is in $A^p_{\alpha}$. Notice that the scopes of these three theorems overlap. When $-1 < \alpha \le 0$, only Theorem~\ref{perp1} applies. When $0 < \alpha \le 1$, Theorem~\ref{perp2} and Theorem~\ref{perp1} both apply with Theorem~\ref{perp2} being the stronger of the two. When $ \alpha  > 1$ all three apply, with Theorem~\ref{H2} being the strongest.

We now look for conditions on zero sequences $\{a_n\}$ that will imply that the derivative of every model space function corresponding to certain Blaschke products, but not necessarily all Blaschke products, is in $A^p_{\alpha}$ for additional values of $(\alpha, p)$.

\begin{theorem} \label{zeros1} Let $B$ be an infinite Blaschke product with  zeros $\{a_n\}$. If $-1 < \alpha \le 0$ and $\frac{2}{3} + \frac{2}{3}\alpha \le p < 1 + \alpha$, and if
\[ \sum_n (1 - |a_n|)^{\frac{p}{2 - p}} < \infty, \]
then $f' \in A^p_{\alpha}$ for all $f \in (BH^2)^{\perp}$.
\end{theorem}
\begin{proof}
By Theorem~\ref{perp1} the conclusion holds for all Blaschke products $B$ when $0 < p <  \frac{2}{3} + \frac{2}{3}\alpha$, which is why we restrict our attention to $p \ge \frac{2}{3} + \frac{2}{3}\alpha$. Let $f \in (BH^2)^{\perp}$. As in the proof of Theorem~\ref{perp2}, equation~(\ref{derivative}) holds. Note that now $0 < p < 1$. Thus, $|x + y|^p \le (|x| + |y|)^p \le |x|^p + |y|^p$ for any $x$ and $y$, and so,
\[ \|f'\|_{A^p_{\alpha}}^p \lesssim \sum_n |c_n|^p (1 - |a_n|^2)^{p/2}I_1(a_n) + \sum_n |c_n|^p (1 - |a_n|^2)^{p/2} I_2(a_n), \]
where $I_1(a_n)$ and $I_2(a_n)$ are as defined in the proof of Theorem~\ref{perp2}. By the Lemma, $I_1(a_n) \asymp 1$ since $\alpha  > -1$ and $2p < \alpha + 2$, while $I_2(a_n) \asymp 1$ since $\alpha - p  > -1$ and $p < \alpha - p + 2$. Thus,
\[\|f'\|_{A^p_{\alpha}}^p \lesssim \sum_n |c_n|^p (1 - |a_n|^2)^{p/2} \le \left(\sum_n |c_n|^2 \right)^{\frac{p}{2}} \left(\sum_n (1 -|a_n|^2)^\frac{p}{2 - p} \right)^\frac{2 - p}{2} < \infty \]
by H\"{o}lder's inequality since $\left( \frac{2}{p}\right)^{-1}\!\!\!+ \left( \frac{2}{2 - p}\right)^{-1}\!= 1$.
\end{proof}

\begin{theorem} \label{zeros2} Let $B$ be an infinite Blaschke product with zeros $\{a_n\}$. If $0 < \alpha \le 1$ and $1 + \frac{1}{2}\alpha < p < 1 + \alpha$ and if
\[ \sum_n (1 - |a_n|)^{\frac{4 - 3p + 2\alpha}{p}} < \infty, \]
then $f' \in A^p_{\alpha}$ for all $f \in (BH^2)^{\perp}$.
\end{theorem}
\begin{proof}
By Theorem~\ref{perp2} the conclusion holds for all Blaschke products $B$ when $\alpha > 0$ and $0 < p < 1 + \frac{1}{2}\alpha$, which is why we restrict our attention to $p \ge 1 + \frac{1}{2}\alpha$.       Let $f \in (BH^2)^{\perp}$. Since $p > 1$, we can proceed as we did in the proof of Theorem~\ref{perp2}. We get 
\[ \|f'\|_{A^p_{\alpha}} \lesssim \sum_n |c_n| (1 - |a_n|^2)^{1/2}[I_1(a_n)]^{1/p} + \sum_n |c_n| (1 - |a_n|^2)^{1/2} [I_2(a_n)]^{1/p}. \]
But $2p > \alpha + 2$, $\alpha > -1$, and $\alpha > -1 + p$. So the Lemma says $I_1(a_n) \asymp \frac{1}{(1 - |a|)^{2p - \alpha - 2}}$ and $I_2(a_n) \asymp \frac{1}{(1 - |a|)^{2p - \alpha - 2}}$. Thus,
\[\|f'\|_{A^p_{\alpha}} \lesssim \sum_n |c_n| (1 - |a_n|^2)^{\frac{4-3p + 2\alpha}{2p}} \le \left(\sum_n |c_n|^2 \right)^\frac{1}{2} \left(\sum_n (1 -|a_n|^2)^{\frac{4-3p + 2\alpha}{p}} \right)^\frac{1}{2} < \infty \]
by H\"{o}lder's inequality.
\end{proof}

Our next step is to see what additional information can be deduced when the sequence of zeros $\{a_n\}$ is assumed to be uniformly separated (or the union of finitely many uniformly separated sequences). In Theorem~4 of \cite{AV}, A.~ Aleman and D.~Vukoti\'{c} proved a result for normal weights which in the case of standard weights says in part that if $\alpha > -1$, $p > 1$,  and $1 + \alpha < p < \frac{4}{3} + \frac{2}{3}\alpha$, then 
\[ \sum_n (1 - |a_n|)^{\frac{4 - 3p + 2\alpha}{p}} < \infty \]
if and only if $f' \in A^p_{\alpha}$ for all $f \in (BH^2)^{\perp}$. The next theorem shows that the scope of Theorem~\ref{zeros1} (with Theorem~\ref{perp1}) can be increased to $-1 < \alpha \le 0$, $0 < p < 1 + \frac{1}{2}\alpha$ when the zeros of $B$ are assumed to be uniformly separated.

\begin{theorem} \label{unifsep1} Let $B$ be an infinite Blaschke product with uniformly separated zeros $\{a_n\}$. If $-1 < \alpha \le 0$ and $1 + \alpha \le  p < 1 + \frac{1}{2}\alpha$ and if
\[ \sum_n (1 - |a_n|)^{\frac{p}{2 - p}} < \infty, \]
then $f' \in A^p_{\alpha}$ for all $f \in (BH^2)^{\perp}$.
\end{theorem}
\begin{proof}
Let $f \in (BH^2)^{\perp}$. Since $\{a_n\}$ is uniformly separated, $f$ can be expressed as $f(z) = \sum_n c_n (1 - |a_n|)^{\frac{1}{2}}/(1 - \bar{a}_n z)$ where $\|f\|_{H^2}^2 \asymp \sum_n |c_n|^2$. So,
 \[ |f'(z)| = \left|\sum_n c_n\frac{\bar{a}_n(1 - |a_n|)^{\frac{1}{2}}}{(1 - \bar{a}_n z)^2}\right| \le \sum_n |c_n|\frac{(1 - |a_n|)^{\frac{1}{2}}}{|1 - \bar{a}_n z|^2}. \]
Now as in the proof of Theorem~6, we use $p$ being less than 1, the Lemma with $2p < \alpha + 2$, and H\"{o}lder's inequality with $\left( \frac{2}{p}\right)^{-1}\!\!\!+ \left( \frac{2}{2 - p}\right)^{-1}\!= 1$. We get
\begin{align*} 
\|f'\|^p_{A^p_{\alpha}} & \le \iint_U \left( \sum_n |c_n| \frac{(1 - |a_n|)^{\frac{1}{2}}}{|1 - \bar{a}_n z|^2} \right)^p (1 - |z|)^{\alpha}\,dA(z)\\ & \le  \sum_n  |c_n|^p (1 - |a_n|)^{\frac{p}{2}}\iint_U  \frac{(1 - |z|)^{\alpha}}{|1 - \bar{a}_n z|^{2p}} \,dA(z)  \\
& \asymp  \sum_n  |c_n|^p (1 - |a_n|)^{\frac{p}{2}}\\
& \le \left(\sum_n |c_n|^2 \right)^{\frac{p}{2}} \left(\sum_n (1 -|a_n|)^\frac{p}{2 - p} \right)^\frac{2 - p}{2} < \infty.
\end{align*}
\end{proof}

\begin{theorem} \label{unifsep2} Let $B$ be an infinite Blaschke product with uniformly separated zeros $\{a_n\}$. If $-1 < \alpha \le 0$ and $1 + \frac{1}{2}\alpha < p < \min\{1, \frac{4}{3} + \frac{2}{3}\alpha\}$, and if
\[ \sum_n (1 - |a_n|)^{\frac{4 - 3p + 2\alpha}{2 - p}} < \infty, \]
then $f' \in A^p_{\alpha}$ for all $f \in (BH^2)^{\perp}$.
\end{theorem}
\begin{proof}
Let $f \in (BH^2)^{\perp}$. Since $\{a_n\}$ is uniformly separated and $p < 1$ as in  Theorem~\ref{unifsep1},
\[ |f'(z)|  \le \sum_n |c_n|\frac{(1 - |a_n|)^{\frac{1}{2}}}{|1 - \bar{a}_n z|^2} \]
and then
\begin{align*} 
\|f'\|^p_{A^p_{\alpha}} & \le  \iint_U \left( \sum_n |c_n| \frac{(1 - |a_n|)^{\frac{1}{2}}}{|1 - \bar{a}_n z|^2} \right)^p (1 - |z|)^{\alpha}\,dA(z) \\
& \le  \sum_n  |c_n|^p (1 - |a_n|)^{\frac{p}{2}}\iint_U  \frac{(1 - |z|)^{\alpha}}{|1 - \bar{a}_n z|^{2p}} \,dA(z).
\end{align*}
But now since $2p > \alpha +2$,
\[ \iint_U \frac{(1 - |z|)^{\alpha}}{|1 - \bar{a}_n z|^{2p}}\,dA(z) \asymp \frac{1}{(1 - |a_n|)^{2p - \alpha - 2}}\,.\]
 Thus,
\begin{align*}
\| f'\|^p_{A^p_{\alpha}} &\lesssim \sum_n |c_n|^p (1 - |a_n|)^{\frac{4-3p + 2\alpha}{2}} \\
&\le \left(\sum_n |c_n|^2 \right)^\frac{p}{2} \left(\sum_n (1 -|a_n|)^{\frac{4-3p + 2\alpha}{2 - p}} \right)^\frac{2 - p}{2} < \infty.
\end{align*}
by H\"{o}lder's inequality since $\left( \frac{2}{p}\right)^{-1}\!\!\!+ \left( \frac{2}{2 - p}\right)^{-1}\!= 1$.
\end{proof}

Figure~\ref{summary} summarizes what the results in this section say about at which points $(\alpha, p)$, $f' \in A^p_{\alpha}$ for all $f \in (BH^2)^{\perp}$. Region~A, being the union of the scopes of Theorems 3, 4, and 5, gives us the points $(\alpha, p)$ at which we have shown that the desired conclusion holds for every infinite Blaschke product $B$. Regions B and C, corresponding to the scopes of Theorems 6 and 7, involve points at which the result holds for certain Blaschke products. Regions D, E and F, corresponding to the scopes of the theorem of Aleman and Vukoti\'{c}, Theorem 8 and Theorem 9, are restricted to interpolating Blaschke products. What happens in the region between $p = \frac{4}{3} + \frac{2}{3}\alpha$ and $p = 2 + \alpha$ is an open question.

\begin{figure}[ht]
\includegraphics{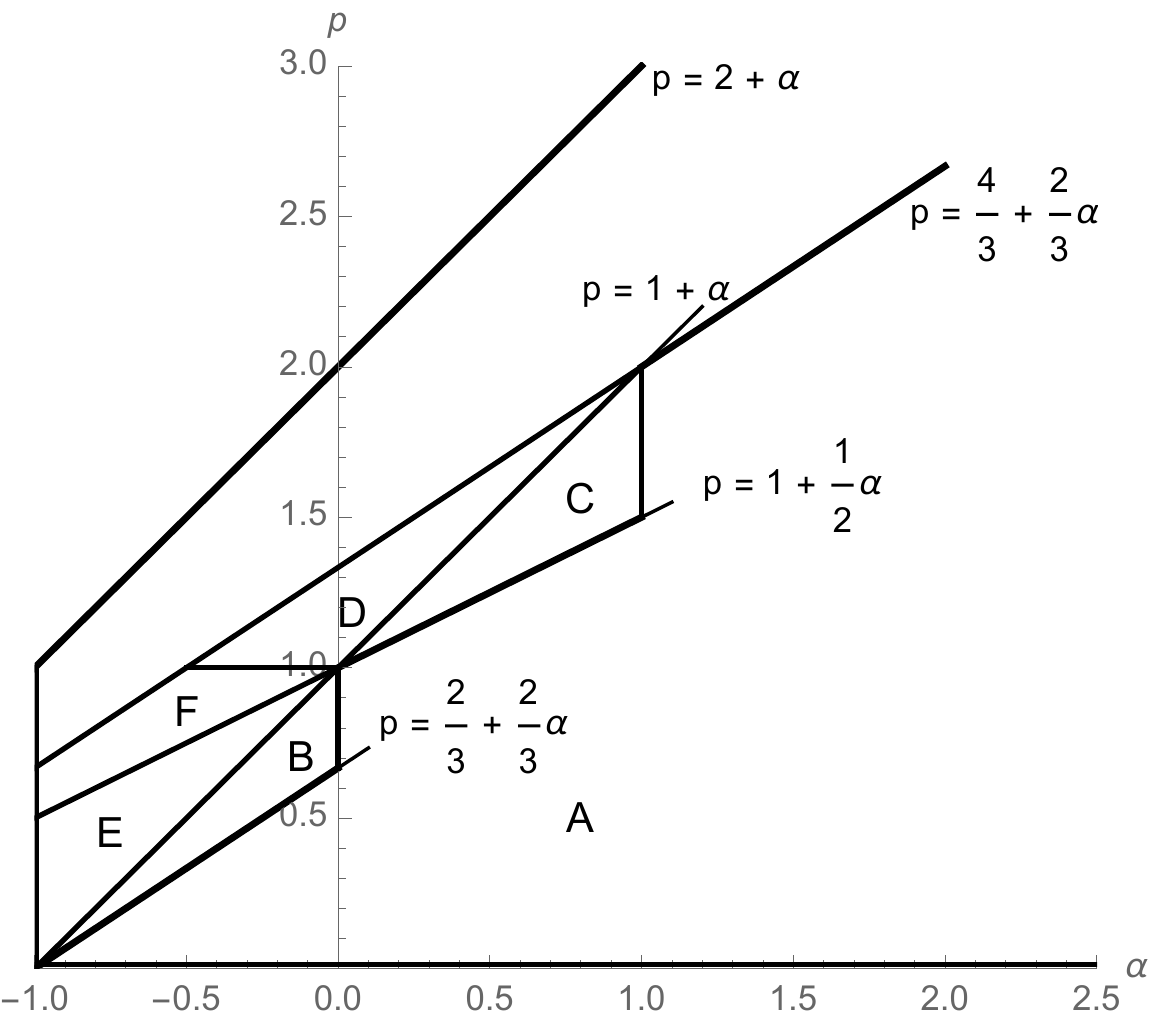}
\caption{Where $f' \in A^p_{\alpha}$ for all $f \in (BH^2)^{\perp}$}
\label{summary}
\end{figure}

We finish with one more way of restricting the zeros $\{a_n\}$ of a Blaschke product. For any $\xi \in \partial U$ and any $\eta > 1$, $\Omega_{\eta}(\xi) = \{z \in U \colon |1 - \bar{\xi} z| \le \eta (1 - |z|) \}$ is a \emph{Stolz domain} or \emph{Stolz angle}. In Theorem 1(c) of \cite{R1} (when applied to the weights $(1 - r)^{\alpha}$), A.~Reijonen, proved that if $B$ is a Blaschke product with zeros $\{a_n\}$ in a Stolz domain, and if $\alpha > -1$ and $\frac{1}{2} < p < \frac{3}{2} + \alpha$, then $B' \in  A^p_{\alpha}$. We prove a similar result for the derivative of model space functions in $A^p_{\alpha}$.

\begin{theorem} \label{stolz} Let $B$ be a Blaschke product with  zeros $\{a_n\}$ in a Stolz domain. If $\alpha > -1$ and $0 < p < 1 + \frac{2}{3}\alpha$, then $f' \in A^p_{\alpha}$ for all $f \in (BH^2)^{\perp}$.
\end{theorem}
\begin{proof}
Let $f \in (BH^2)^{\perp}$. As in the proof of Theorem~\ref{perp1}, $|f(z)| \lesssim  \left( \frac{1 - |B(z)|^2}{1 - |z|^2} \right)^{\frac{1}{2}}$ for all $|z| < 1$, and so
\[ M_p(r; f) \lesssim \left\{\int_0^{2\pi} \left( \frac{1 - |B(re^{i\theta})|^2}{1 - r^2} \right)^{\frac{p}{2}}\,d\theta\right\}^{\frac{1}{p}}. \]
But, Reijonen proved in Proposition~3.2(ii) of \cite{R2} that $\int_0^{2\pi} \left( \frac{1 - |B(re^{i\theta})|}{1 - r} \right)^{\frac{p}{2}}\,d\theta \lesssim (1 - r)^{\frac{1}{2} - \frac{p}{2}}$ for $\{a_n\}$ being contained in a Stolz domain and $p > 1$. Then $M_p(r; f) \lesssim (1 - r)^{\frac{1}{2p} -\frac{1}{2}}$, and so $M_p(r; f') \lesssim (1 - r)^{\frac{1}{2p} -\frac{3}{2}}$ by \cite[Theorem 5.5]{D}. Therefore, $\|f'\|^p_{A^p_{\alpha}} \lesssim \int_0^1 (1 - r)^{\frac{1}{2} -\frac{3}{2}p + \alpha}\,dr < \infty$ since $p < 1 + \frac{2}{3}\alpha$.
\end{proof}

\end{document}